\pdfoutput=1
\documentclass{article}
\usepackage[T1]{fontenc}
\usepackage{amsmath}
\usepackage{amsthm}
\usepackage{newtxtext}
\usepackage[varvw]{newtxmath}
\usepackage{graphicx}
\usepackage{booktabs}
\usepackage{algorithm}
\usepackage{algpseudocode}
\usepackage[hyphens]{url}
\newcommand{\eb}{\overline}
\DeclareMathOperator{\ee}{e}
\makeatletter
\newcommand{\setcitename}[2]{\@namedef{citeauthors@#1}{#2}}
\newcommand{\citeauthors}[1]{\@ifundefined{citeauthors@#1}%
  {\@latex@warning{No author name for citation `#1'}\textbf{??}}%
  {\@nameuse{citeauthors@#1}}}
\makeatother
\newcommand{\citep}[1]{\cite{#1}}
\newcommand{\citet}[1]{\citeauthors{#1}~\cite{#1}}
\newcommand{\citealp}[1]{\citeauthors{#1}~\cite{#1}}
\setcitename{EG1999}{Erd\H{o}s and Gy\'arf\'as}
\setcitename{KirchwegerManriqueSzeider}{Kirchweger, Manrique, and Szeider}
\setcitename{LamThielSwiercz1989}{Lam, Thiel, and Swiercz}
\setcitename{McKay1998}{McKay}
\setcitename{Brouwer1981}{Brouwer}
\setcitename{KangPikhurko2005}{Kang and Pikhurko}
\setcitename{ChungLiu1978}{Chung and Liu}
\usepackage{listings}
\usepackage{xcolor}
\lstdefinelanguage{lean4}{
  morekeywords={def,theorem,lemma,where,fun,by,open,namespace,end,variable,instance,
    noncomputable,Prop,Type},
  sensitive=true, morecomment=[l]{--}, morecomment=[s]{/-}{-/}, morestring=[b]",
}
\usepackage{worldflags}
\makeatletter
\g@addto@macro\@floatboxreset{\small}
\makeatother
\newtheorem{theorem}{Theorem}
\newtheorem{lemma}{Lemma}
\newtheorem{corollary}{Corollary}
\theoremstyle{remark}
\newtheorem{remark}{Remark}
\makeatletter
\newenvironment{credits}{\par\small\emergencystretch=1em
  \renewcommand\subsubsection[1]{\@startsection{subsubsection}{3}{\z@}%
    {-12\p@ \@plus -4\p@ \@minus -4\p@}%
    {-0.5em \@plus -0.22em \@minus -0.1em}%
    {\normalfont\small\bfseries}*{##1}}}{\par}
\makeatother
\usepackage{hyperref}
\begin{document}

\title{Erd\H{o}s--Gy\'arf\'as Balanced Five-Colorings}
\author{Hai Xia\qquad Stefan Szeider\\[0.75ex]
  {\normalsize Algorithms and Complexity Group, TU Wien, Vienna, Austria}\\
  {\normalsize\texttt{\{hxia,sz\}@ac.tuwien.ac.at}}}
\date{}
\maketitle

\begin{abstract}
A coloring of the edges of a complete graph with $r$ colors is balanced if the edges among any $r+1$ vertices use all $r$ colors. Erd\H{o}s and Gy\'arf\'as (1999) conjectured that for all $r\ge3$ the complete graph $K_{r^2+1}$ has no balanced $r$-coloring. They proved the conjecture for $r=3$ and $r=4$. We prove the conjecture for $r=5$, the first new case settled in over two decades. There is no balanced $5$-coloring of $K_{26}$, and since $K_{25}$ has one, $K_{25}$ is the largest complete graph admitting a balanced $5$-coloring. We design a reduce--dispatch--enumerate workflow, where every machine step produces an artifact that can be independently checked. First, we use a machine-checked (Lean~4) proof to reduce the conjecture to a single extremal statement about the color class with the fewest edges: no graph on $26$ vertices has between $55$ and $65$ edges while every $6$ of its vertices span between $1$ and $11$ edges. We then search on the sparse class instead of its dense complement, which makes enumeration up to isomorphism feasible. With classical stability theorems for $K_6$-free graphs, we dispatch the five near-extremal edge counts analytically. For the exhaustive enumeration on the remaining six edge counts, we design efficient encodings with SAT modulo symmetries and parallel cube-and-conquer. We also certify all eleven edge counts: each of the $12{,}296$ subproblems carries a clausal refutation checked by the independent verifier \texttt{drat-trim}, and all $273{,}503{,}938$ symmetry-breaking clauses are also validated against the vertex permutations that justify them.
\end{abstract}
\section{Introduction}\label{sec:intro}

For an integer $r\ge2$, a \emph{balanced $r$-coloring} of the complete graph $K_n$ on $n$ vertices is a map
$c\colon E(K_n)\to\{1,\dots,r\}$ from its edges to colors such that every set of $r+1$ vertices
\emph{spans} edges of all $r$ colors, that is, the edges with both ends in the set use all $r$ colors. Balanced colorings were introduced in 1999 by \citet{EG1999},
who also gave a reformulation: a balanced $r$-coloring of $K_n$ exists if and only if $E(K_n)$
can be partitioned into $r$ graphs $G_1,\dots,G_r$ with $\alpha(G_i)\le r$ for each $i$. Here a
set of vertices of a graph is \emph{independent} if it spans no edge, and the
\emph{independence number} $\alpha(G)$ of a graph $G$ is the largest size of an independent
set. In detail, an $r$-coloring and a partition of
$E(K_n)$ into $r$ graphs determine each other, with $G_i$ the graph on all $n$ vertices whose edges
are those of color~$i$. A set of $r+1$ vertices misses color~$i$, that is, none of the edges among them has
color~$i$, exactly when it is an independent set of~$G_i$. Hence the coloring is balanced
exactly when no $G_i$ has an independent set of $r+1$ vertices, which is the condition
$\alpha(G_i)\le r$ for every~$i$.

\citet{EG1999} conjectured a sharp threshold for the existence of balanced colorings.

\begin{quote}
\textbf{Conjecture (Erd\H{o}s--Gy\'arf\'as).} For every $r\ge3$ there is no balanced $r$-coloring of
$K_{r^2+1}$.
\end{quote}

The value $r^2+1$ is not an arbitrary threshold. When $r$ is a prime power (that is, $r=p^{\ell}$ for a
prime~$p$ and an integer $\ell\ge1$), the finite field with $r$ elements (for prime $r$, the integers
modulo~$r$) yields an affine plane of order~$r$, which gives a balanced $r$-coloring of $K_{r^2}$. The plane has $r^2$ points and lines of $r$ points each. Every two points lie on exactly one line, and the lines fall into
$r+1$ parallel classes of $r$ pairwise disjoint lines. Take the points as the vertices of $K_{r^2}$
and put each edge into the parallel class of the line through its two ends. This splits the edges
into $r+1$ classes, each consisting of $r$ disjoint cliques (sets of pairwise adjacent vertices) on $r$ vertices~\citep{EG1999}. An
independent set meets each clique in at most one vertex, so every class has independence number~$r$.
Merging two classes into one color leaves $r$ colors, each still with independence number at
most~$r$. Therefore, the coloring is balanced (Fig.~\ref{fig:examples}, right). The construction is called the \emph{affine coloring} of~$K_{r^2}$. The Erd\H{o}s--Gy\'arf\'as conjecture asserts that we cannot push past $r^2$. For $r=2$ it is false: coloring the five edges of a pentagon $C_5$ with one color and the other five
edges of $K_5$ with the second gives a balanced $2$-coloring of $K_5$ (Fig.~\ref{fig:examples}, left). For $r=3$ and $r=4$, \citet{EG1999} proved it: there is no balanced $3$-coloring of $K_{10}$ and no balanced $4$-coloring of $K_{17}$. Their two proofs are case-specific. The $r=3$ argument relies on Brooks' theorem~\citep{Brooks1941} and a short case analysis; this case had in fact been settled already in 1978 by \citet{ChungLiu1978} in the language of generalized Ramsey numbers (their Theorem~3.6): every $3$-coloring of $K_{10}$ contains a $K_4$ that uses at most two colors. The $r=4$
argument involves edge counting and the Ramsey number $R(3,4)=9$~\citep{GreenwoodGleason1955}. The case $r=5$, the existence of a balanced $5$-coloring of $K_{26}$, has remained open since 1999.

\begin{figure}[h]
\centering
\includegraphics{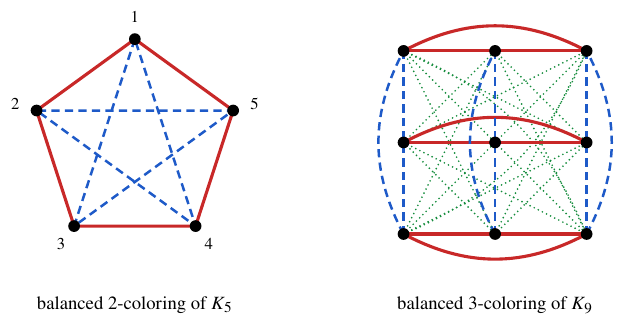}%
\caption{Two balanced colorings. Left: $K_5$ colored by the pentagon $C_5$ (solid) and its
complement (dashed); every $3$ vertices span both colors, which is why the conjecture is stated
for $r\ge3$. Right: the affine coloring of $K_9$. The vertices are the $9$ points of the
affine plane of order~$3$, drawn as a $3\times3$ grid, whose lines are the rows, the
columns, and two classes of diagonals that wrap around the grid. Edges inside a row get the
first color (solid), edges inside a column the second (dashed), and all remaining edges, which
form the two diagonal classes, the third (dotted). Any $4$ vertices include two in the same
row, two in the same column, and two in neither, so every $4$ vertices span all three colors.}
\label{fig:examples}
\end{figure}

In this paper, our main result settles Erd\H{o}s--Gy\'arf\'as balanced five-colorings as stated in Theorem~\ref{thm:main}.

\begin{theorem}\label{thm:main}
There is no balanced $5$-coloring of $K_{26}$. Equivalently, every $5$-coloring of the edges of $K_{26}$
contains a copy of $K_6$ that misses at least one color. The Erd\H{o}s--Gy\'arf\'as conjecture holds for
$r=5$.
\end{theorem}

\begin{corollary}\label{cor:threshold}
The maximum $n$ for which $K_n$ admits a balanced $5$-coloring is exactly $25=5^2$.
\end{corollary}

\begin{proof}
Restricting a balanced $5$-coloring of $K_n$ to any $26$ of its vertices would give a balanced
$5$-coloring of $K_{26}$, since every set of $6$ chosen vertices still spans all $5$ colors.
Theorem~\ref{thm:main} therefore rules out every $n\ge26$. The affine coloring of $K_{25}$ described
above attains $25$.
\end{proof}

A balanced
$5$-coloring of $K_{26}$ can be written as one propositional formula: the Boolean variables indicate the colors of the edges, and the clauses require that each edge has exactly one color and that every set of $6$ vertices contains an edge of each color. However, such formulas are not solved in practice by state-of-the-art satisfiability (SAT) or
constraint solving. Even the simpler formula for the \emph{known} case $r=4$ on $K_{17}$ was not
decided in our tests within $3600$ seconds by the SAT solver kissat~\citep{BiereKissat2020}, by the
constraint solver CP-SAT \citep{PerronDidierGay2023}, or by CP-SAT with added cardinality constraints.
The obstacle is the huge search space with its symmetries: relabeling the $17$ vertices and permuting the $4$ colors turns every coloring into up to $17!\cdot4!$ equivalent copies (the group $S_{17}\times S_4$, where $S_m$ permutes $m$ objects). To overcome this obstacle, we design a reduce--dispatch--enumerate workflow: a reduction, an analytic dispatch of the near-extremal cases, and a computation. Furthermore, we design layers of certification to ensure that our results are trustworthy.

\paragraph{The reduction.}
We write $\ee(G)$ for the number of edges of a graph $G$. The smallest color class $G_0$ of a balanced
$5$-coloring of $K_{26}$ has $\ee(G_0)\le325/5=65$, as the five classes partition the $325$ edges. On the other hand, with
$\alpha(G_0)\le5$ and the Caro--Wei inequality (a lower bound on the independence number in terms of the
vertex degrees, see Section~\ref{sec:reduction}), we have $\ee(G_0)\ge55$. These two bounds leave the $11$ edge
counts $e\in[55,65]$. Furthermore, the definition of \emph{balance} also forces a local-density property (L), and the conjecture can be further reduced to the following statement, by a proof formally verified in the proof assistant Lean~4 (Section~\ref{sec:reduction}):

\begin{quote}
\textbf{Main Lemma.} No graph on $26$ vertices has independence number at most $5$, the local-density
property \textup{(L)} that every $6$ vertices span at most $11$ edges, and between $55$ and $65$ edges.
\end{quote}
We settle the five edge counts $e\in[55,59]$ analytically with two classical stability theorems~\citep{Brouwer1981,KangPikhurko2005}, as the complement of $G_0$ then has at least $266$ of the at most $270$ edges of a $K_6$-free graph on $26$ vertices (Section~\ref{sec:analytic}).

\paragraph{The computation.}
The remaining $6$ edge counts, $e\in[60,65]$, are settled by an exhaustive enumeration of the graphs for
the smallest class $G_0$ rather than for its dense complement. We enumerate the graphs with \emph{SAT modulo symmetries} (SMS)~\citep{KirchwegerSzeider}, using encodings designed for it, where a SAT solver prunes graphs that are not the lexicographically smallest copy in their isomorphism class (Section~\ref{sec:prelim}). We further parallelize the search by cube-and-conquer, choosing the variables to fix by their frequency (Section~\ref{sec:cubevars}), so that the computational load is spread evenly over independent subproblems (\emph{cells}). For $e\in\{63,64,65\}$, we split each formula into $4{,}096$ cubes, and we obtain refutation proofs for all subproblems. To make the computation trustworthy, we not only checked the proofs with an independent checker, but also validated all symmetry-breaking clauses added during the search against their \emph{permutation witnesses}, the
relabelings of the vertices that justify them (Section~\ref{sec:cert}). Fig.~\ref{fig:roadmap} summarizes the workflow.

\begin{figure}[h]
\centering
\includegraphics[width=0.8\columnwidth]{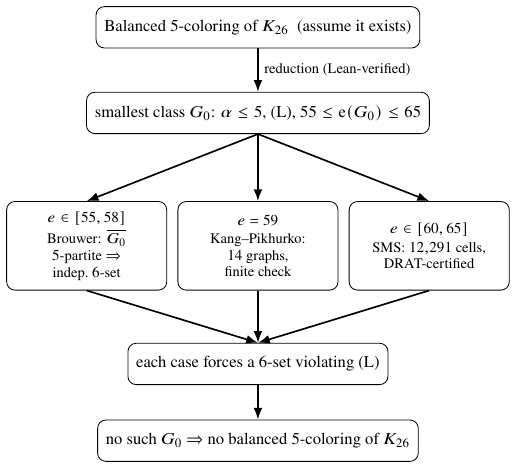}
\caption{The proof at a glance. A balanced $5$-coloring of $K_{26}$ would yield a smallest color class
$G_0$ with $\alpha\le5$, the local-density property (L), and $55\le\ee(G_0)\le65$ (Lemma~\ref{lem:band}), which the Main Lemma excludes. Two
classical stability theorems dispatch the near-extremal counts $e\in[55,59]$ analytically (a graph is $5$-partite if its
vertices split into $5$ independent sets), and a
\texttt{drat-trim}-certified SAT-modulo-symmetries enumeration settles $e\in[60,65]$. Each of the three cases
forces a $6$-set violating (L), so $G_0$ cannot exist.}
\label{fig:roadmap}
\end{figure}

\paragraph{Contributions.}
(i)~\emph{Result.} We settle the case $r=5$ of the Erd\H{o}s--Gy\'arf\'as conjecture, the first new
case since 1999 (Theorem~\ref{thm:main}, Corollary~\ref{cor:threshold}).
(ii)~\emph{A verified reduction chain.} A reduction checked in Lean~4 turns the conjecture into the
Main Lemma, a statement about one sparse graph with $11$ possible edge counts. The formal main theorem
has the Main Lemma as its only hypothesis.
(iii)~\emph{Efficient encodings and cube selection.} We encode the sparse class, impose the
local-density property only on one densest $6$-set per enumerated graph, and choose the $12$ cube
variables by their frequency, so that the heaviest cube holds $8.3\%$ of a test family, against $74$ to
$100\%$ for structural choices and SMS's default cubing (Section~\ref{sec:computation}).
(iv)~\emph{Layers of certification.} Programs separate from the search check the encodings, the
completeness of the enumeration, the refutations of all $12{,}296$ cells (with \texttt{drat-trim}), and
all $273{,}503{,}938$ symmetry-breaking clauses against their permutation witnesses. As the certified
search covers all $11$ edge counts, the result depends on neither external theorem (Section~\ref{sec:cert}, Table~\ref{tab:trust}).

\section{Preliminaries}\label{sec:prelim}

All graphs in our paper are finite and simple (no loops, no multiple edges between vertices). For a graph $G$, we denote its number of edges and independence number with $\ee(G)$ and
$\alpha(G)$, respectively.

The \emph{clique number} $\omega(G)$ is the largest number of pairwise adjacent vertices. A graph is \emph{$K_t$-free} if it has no $t$ pairwise adjacent vertices. Therefore, $K_6$-free is equivalent to $\omega(G)\le5$.

The \emph{chromatic number} $\chi(G)$ is the smallest number of colors in a coloring of the vertices such that adjacent vertices receive different colors. A graph is \emph{$r$-partite} when $\chi(G)\le r$. That is, the vertices of an $r$-partite graph can be split into $r$ independent sets, its \emph{parts}.

The \emph{complement} graph of $G$, $\eb G$, has the same
vertices as $G$ and contains exactly the pairs that are not edges of $G$. Then we have $\omega(\eb G)=\alpha(G)$ and $\alpha(\eb G)=\omega(G)$. Throughout our paper for $K_{26}$, we have the number of vertices, $N=26$ and the number of edges, $\binom{N}{2}=325$. For integers $a\le b$, $[a,b]$ denotes the range,
$\{a,a+1,\dots,b\}$. When the graph is clear from context, we write $\alpha\le5$ for $\alpha(G)\le5$, and similarly for $\omega$ and $\ee$.

A graph is \emph{complete $r$-partite} if its vertices split into $r$ parts with no edge inside a part and an edge between every two vertices in different parts. $K_{n_1,\dots,n_r}$ denotes the one with parts of sizes $n_1,\dots,n_r$. The Tur\'an graph $T(n,r)$ is the complete $r$-partite graph on
$n$ vertices with parts as equal as possible. For our instance, $n=26$ and $r=5$, the corresponding parts have sizes $(6,5,5,5,5)$, so $t_5(26):=\ee(T(26,5))=\ee(K_{6,5,5,5,5})=270$. By Tur\'an's theorem~\citep{Turan1941}, a $K_6$-free graph on $26$ vertices has at most $270$ edges, with equality only for $T(26,5)$. The following two stability theorems describe the $K_6$-free graphs just below this bound.

\begin{theorem}[\citealp{Brouwer1981}]\label{thm:brouwer}
Every $K_6$-free graph on $26$ vertices with at least $t_5(26)-\lfloor 26/5\rfloor+2=267$ edges is
$5$-partite.
\end{theorem}

\begin{theorem}[\citealp{KangPikhurko2005}]\label{thm:kp}
The maximum number of edges of a $K_6$-free graph on $26$ vertices that is not $5$-partite is exactly
$266$. Up to isomorphism, the extremal graphs are exactly
\[
C_5[1,q,\,n_2{-}1,\,n_1{-}q,\,1]\ \vee\ K_{n_3,n_4,n_5},
\]
where $C_5[a_1,\dots,a_5]$ is the graph obtained from the $5$-cycle by replacing its $i$-th vertex with an independent set of $a_i$ vertices, with all edges between consecutive sets, $\vee$ is the join, which adds all edges between the two graphs, $1\le q\le n_1-1$, and the tuple of part sizes $(n_1,\dots,n_5)$ is one of $(4,4,5,6,6)$, $(4,5,5,5,6)$, $(5,4,5,5,6)$ and $(5,5,5,5,5)$.
\end{theorem}

These tuples are exactly the ones for which the construction is $K_6$-free, not $5$-partite, and has $266$ edges. Every part of the pentagon skeleton and of the tripartite part $K_{n_3,n_4,n_5}$ is an independent
set. Fig.~\ref{fig:kp} shows an example of the construction.

\begin{figure}[h]
\centering
\includegraphics{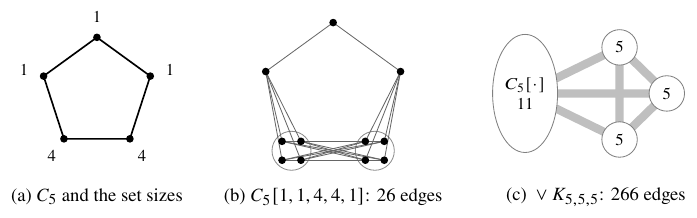}%
\caption{Construction of the extremal graph $C_5[1,1,4,4,1]\vee K_{5,5,5}$ of \citealp{KangPikhurko2005} for the tuple
$(5,5,5,5,5)$ and $q=1$. (a)~The $5$-cycle skeleton: each vertex labeled with the size of the set that replaces it.
(b)~The blow-up: every set is independent, and consecutive sets are joined by all edges, $26$ in total.
(c)~The join with $K_{5,5,5}$: a gray band stands for all edges between two parts. The result has $26$ vertices
and $266$ edges.}
\label{fig:kp}
\end{figure}

\paragraph{SAT modulo symmetries.}
The computational core enumerates isomorph-free graphs subject to combinatorial constraints. As we mentioned in the introduction, encoding the constraints with one Boolean variable per possible edge and running a SAT solver to enumerate models fails at our scale: the vertex-relabeling group $S_{26}$ has $26!$ elements. Even though static symmetry-breaking constraints~\citep{Crawford1996} can be added to the formula before solving for excluding some symmetric copies, they can remove only part of this symmetry in practice.
\emph{SAT modulo symmetries} (SMS)~\citep{KirchwegerSzeider} interleaves a conflict-driven clause-learning (CDCL) SAT solver with a propagator. SMS compares adjacency matrices lexicographically, as $0/1$ strings of their entries above the
diagonal read row by row, and calls a graph \emph{canonical} if its matrix is the smallest in its
isomorphism class. The propagator runs a \emph{minimality check} at regular intervals: it looks for a vertex permutation showing that no completion of the current partial adjacency matrix is canonical, and if it finds one, it learns a \emph{symmetry-breaking clause} (\emph{symmetry clause} for short) that prunes the branch. When the constraints are invariant under relabeling the vertices, a complete run of SMS emits at least one representative of every isomorphism class of graphs satisfying them.

We use the official implementation of SMS\footnote{\url{https://github.com/markirch/sat-modulo-symmetries}}.
Two of its features matter here. First, SMS bounds the work of each minimality check by a \emph{cutoff} on its number of
recursive steps ($200{,}000$ by default). A check that reaches the cutoff keeps
the graph, so in its enumerate-all mode with a cutoff SMS may emit several isomorphic copies
of a class, and its output contains at least one representative of every isomorphism class. This is
exactly what a nonexistence conclusion needs. Second, SMS supports cube-and-conquer. A \emph{cube} is a conjunction of literals that fixes the
truth values of chosen variables, and the $2^k$ cubes over $k$ chosen variables split the formula into
disjoint subproblems whose union is the whole search space.
A proof-logging SMS also emits the symmetry clauses along with a \emph{permutation witness}. The witness contains a vertex permutation showing that every graph the clause excludes has a lexicographically smaller relabeling. Therefore, the clause excludes no canonical graph (Section~\ref{sec:cert}). The
formula together with these clauses, the \emph{augmented formula}, can then be refuted by an ordinary
SAT solver and the refutation checked by an independent verifier~\citep{KirchwegerManriqueSzeider}, which verifies in Lean the encoding, the symmetry clauses and the refutation of a single formula. That framework does not
yet support cube-and-conquer. Our refutations total
$5.5$\,TB in the DRAT format (Table~\ref{tab:counts}). We therefore use only its proof-logging
SMS build, verify the refutations with \texttt{drat-trim}, and check the symmetry clauses
with our own checker (Section~\ref{sec:cert}).

\section{Reduction to a Single Color Class}\label{sec:reduction}

For contradiction, we suppose that a balanced $5$-coloring of $K_{26}$ exists, with color classes
$G_1,\dots,G_5$ partitioning the $325$ edges. Every set $S$ of $6$ vertices, a \emph{$6$-set}, spans
$\binom{6}{2}=15$ edges, which use all $5$ colors. Hence, for every color class $G_i$:
\begin{itemize}
\item[\textbf{(A)}] $\alpha(G_i)\le5$: A $6$-set without an edge of $G_i$ would miss color~$i$.
\item[\textbf{(L)}] Every $6$-set spans at most $11$ edges of $G_i$: The other four colors each occupy at
least one of the $15$ edges, so at most $15-4=11$ are of color~$i$. In particular $\omega(G_i)\le5$, since a
$K_6$ would be a $6$-set spanning all $15$ edges in one color.
\end{itemize}
Properties (A) and (L) hold for \emph{every} color class, and the smallest class controls the edge
count.

\begin{lemma}\label{lem:band}
Some color class $G_0$ of a balanced $5$-coloring of $K_{26}$ satisfies $\alpha(G_0)\le5$, property
\textup{(L)}, and $55\le\ee(G_0)\le65$.
\end{lemma}

\begin{proof}
Take $G_0$ with the fewest edges. As the five classes partition $325$ edges, $\ee(G_0)\le325/5=65$. For the lower bound, the Caro--Wei inequality \citep{Caro1979,Wei1981} gives
$\alpha(G_0)\ge N/(\bar d+1)$ with $\bar d=2\ee(G_0)/N$ the average degree. With $\alpha(G_0)\le5$ and
$N=26$ this forces $\bar d\ge 21/5$, hence $\ee(G_0)=N\bar d/2\ge 54.6$, so $\ee(G_0)\ge55$.
\end{proof}

Lemma~\ref{lem:band} reduces Theorem~\ref{thm:main} to the Main Lemma stated in the introduction. If
the Main Lemma holds, the graph $G_0$ cannot exist, so the assumed coloring cannot exist. The rest of
the paper proves the Main Lemma.

\paragraph{The complement view.}
For a candidate $G$ as in the Main Lemma, with $e=\ee(G)$ edges, set $H=\eb G$.
Then $H$ is $K_6$-free, since $\omega(H)=\alpha(G)\le5$, and has $\ee(H)=325-e\in[260,270]$.
Property (L) becomes a lower bound on local density: every $6$ vertices span at least $15-11=4$ edges of
$H$. In particular $\alpha(H)\le5$. We prove the Main Lemma by ruling out, for each $e\in[55,65]$, a
$K_6$-free graph $H$ on $26$ vertices with $\ee(H)=325-e$ in which every $6$-set spans at least $4$
edges.

\begin{remark}\label{rem:Lneeded}
Property (L) is not implied by $\alpha\le5$ and $\omega\le5$ alone, so property (L) excludes graphs that $\alpha\le5$
and $\omega\le5$ allow. For a witness, we take $H_0=C_5[1,1,4,4,1]\vee K_{5,5,5}$. The graph has $266$ edges
and $\alpha(H_0)=\omega(H_0)=5$, so its complement $G$ has $\ee(G)=59$, and deleting up to six suitable edges from $H_0$ covers $e\in[60,65]$. Each of these $H$ has a $6$-set spanning at most $3$ edges, so the corresponding $G$ has a $6$-set spanning at least $12$ and violates (L).
\end{remark}

\paragraph{The reduction is formally verified.}
The argument of this section is mechanically checked in Lean~4 (Section~\ref{sec:cert}, Listing~\ref{lst:contract}): the implications
``balanced $\Rightarrow$ (A)'' and ``balanced $\Rightarrow$ (L)'', the bound ``some class has at most
$65$ edges'', and the lower bound of $55$ edges are all proved. Lean derives the lower bound from the
formalization of Tur\'an's theorem in Mathlib, applied to the complement.

\section{The Analytic Edge Counts}\label{sec:analytic}

We prove the Main Lemma for $e\in[55,59]$ analytically. Throughout, $H=\eb G$ is $K_6$-free with
$\ee(H)=325-e$, and we look for a $6$-set of $H$ spanning at most $3$ edges, which violates (L).

\paragraph{Edge counts $55$ to $58$.}
Here $\ee(H)\ge 267$, so by Brouwer's theorem (Theorem~\ref{thm:brouwer}) $H$ is $5$-partite. A
$5$-partition of $26$ vertices has a part of size at least $\lceil 26/5\rceil=6$, an independent $6$-set
of $H$ spanning $0$ edges, violating (L).

\paragraph{Edge count $59$.}
Here $\ee(H)=266$. If $\chi(H)\le5$, then as above some part of a $5$-partition of $H$ has at least $6$ vertices, an
independent $6$-set. Otherwise $H$ is $K_6$-free and not $5$-partite, so by Kang--Pikhurko (Theorem~\ref{thm:kp}) it
is one of the explicit graphs $C_5[1,q,n_2{-}1,n_1{-}q,1]\vee K_{n_3,n_4,n_5}$, and we make these
explicit. Fix one of the four admissible size tuples $(n_1,\dots,n_5)$ and a complete $5$-partite graph
$K_{n_1,\dots,n_5}$ with parts $V_1,\dots,V_5$. Pick a vertex $y\in V_2$ and a nonempty proper subset
$A\subsetneq V_1$, delete every edge between $y$ and $A$, and add one new vertex $v$ adjacent to
$V_3\cup V_4\cup V_5\cup\{y\}\cup A$. The resulting graph has $26$ vertices and $266$ edges, is
$K_6$-free, and is not $5$-partite. The parts $V_3$, $V_4$, $V_5$ can be permuted and vertices within a part are
interchangeable, so the choice of $y$ and of $A$ of a given size does not matter up to isomorphism: the construction is determined by the size tuple together with
$q=|A|$, $1\le q\le n_1-1$. No operation adds an edge inside a part, so each $V_i$ remains independent.
This gives fourteen candidate graphs, possibly with isomorphic repetitions (reflecting the $5$-cycle maps the graph for $q$
to the graph for $n_1{-}q$): $q\in\{1,2,3\}$ for $(4,4,5,6,6)$ and $(4,5,5,5,6)$, and
$q\in\{1,2,3,4\}$ for $(5,4,5,5,6)$ and
$(5,5,5,5,5)$. For each of the fourteen candidate graphs we minimize, over all $\binom{26}{6}$ $6$-sets, the number
of $H$-edges spanned. The minimum is $0$ for the ten graphs of the three tuples with a part of size $6$,
since that part is an independent $6$-set, and $1$ for the four graphs of $(5,5,5,5,5)$,
so in every case $H$ has a $6$-set spanning at most $3$ edges and $G$ violates (L).
As an independent check, a SAT solver confirms that none of the fourteen complements extends to a
balanced $5$-coloring of $K_{26}$ (the fourteen graphs and the check are in the supplementary archive, Section~\ref{sec:cert}).
Hence the Main Lemma holds for $e\in[55,59]$.

\section{The Computational Core}\label{sec:computation}

We continue to prove the Main Lemma for $e\in[60,65]$, equivalently, in the complement view, for $\ee(H)\in[260,265]$. We settle each of the $6$ edge counts separately by an exhaustive graph enumeration covering all isomorphism classes. Algorithm~\ref{alg:pipeline} gives the pipeline in outline: the search of this section together with the certification of Section~\ref{sec:cert}.

\begin{algorithm}[h]
\caption{The certified search at edge count $e$, run in full for the six searched counts
$e\in[60,65]$. The five analytic counts $e\in[55,59]$ run the same loop on a single monolithic cell each
(Section~\ref{sec:cert}). The outcome is $z_e=0$ for all
eleven counts, with every certificate checked.}
\label{alg:pipeline}
\begin{algorithmic}[1]
\Require a target edge count $e$ for graphs on $26$ vertices
\Ensure the property (L) survivor count $z_e$; one checked refutation per cell
\State $F_e\gets\textsc{Encode}(\alpha\le5,\ \omega\le5,\ \ee(G)=e)$ \Comment{compared with an independent regeneration}
\State $B\gets$ the $12$ edge variables whose frequency in a saved family is nearest $\tfrac12$ (Sec.~\ref{sec:cubevars})
\State $\mathcal C\gets\{F_e\wedge b\mid b\in\{0,1\}^{B}\}$ if $e\ge63$, else $\{F_e\}$ \Comment{cubes or monolithic}
\State $z_e\gets0$
\ForAll{cells $C\in\mathcal C$} \Comment{one cluster node task each}
  \State $\mathcal G\gets\textsc{Enumerate}(C)$ \Comment{SMS; completeness checked (Sec.~\ref{sec:cert})}
  \State $z_e\gets z_e+|\{G\in\mathcal G\mid G\text{ satisfies (L)}\}|$
  \State $(\Sigma,W)\gets\textsc{ProofLog}(C\wedge L_D)$, with $D$ read off $\mathcal G$ \Comment{symmetry clauses, witnesses}
  \State refute $C\wedge L_D\wedge\Sigma$ with the SAT solver CaDiCaL; check by \texttt{drat-trim}
  \State check each clause of $\Sigma$ against the lex-leader inequality of its witness in $W$ (Sec.~\ref{sec:cert})
\EndFor
\State \Return $z_e$ \Comment{$z_e=0$ for all eleven $e$ is the Main Lemma}
\end{algorithmic}
\end{algorithm}

\subsection{The Encoding}\label{sec:encoding}

Fix the vertex set $V=\{0,\dots,25\}$ and number the $325$ pairs $\{u,v\}$ with $u<v$ in
lexicographic order, so that $x_1,x_2,\dots,x_{325}$ stand for the pairs $\{0,1\},\{0,2\},\dots,
\{24,25\}$. An assignment to these variables is a graph $G$ on $V$: the variable of a pair is true
if the pair is an edge of $G$. For a $6$-set $S\subseteq V$, let $X(S)$ be the set of the
$15$ variables of pairs inside $S$. For a target edge count $e$, the \emph{base formula} $F_e$ is the
conjunction of three clause groups.
\begin{itemize}

  \item[(I)] \emph{Independence.} For every $6$-set $S$ the clause $\bigvee_{x\in X(S)}x$: some edge lies inside $S$. This clause has length $15$, and there are $\binom{26}{6}=230{,}230$ clauses in all. This group forces $\alpha(G)\le5$.

\item[(K)] \emph{Clique.} For every $6$-set $S$, some edge is missing inside $S$; the clause is $\bigvee_{x\in X(S)}\neg x$, again of length $15$. This group forces $\omega(G)\le5$.

\item[(E)] \emph{Edge count.} We use the sequential counter \citep{Sinz2005} encoding over edge variables, $x_1,\dots,x_{325}$. The registers are defined as $s_{i,j}$ for $1\le i\le325$ and $1\le j\le e$. $s_{i,j}$ is true when at least $j$ of the first $i$ pairs are edges. We write $p_{i,j}$ for the literal expressing that at least $j$ of the first $i$ pairs are edges: $p_{i,j}=s_{i,j}$, except that $p_{1,1}=x_1$ and $p_{i,0}$ is the constant true. The register $s_{1,1}$ occurs in no clause. For $2\le i\le325$ and $1\le j\le e$ the clauses
  \[
  \neg p_{i-1,j}\vee s_{i,j},\quad \neg s_{i,j}\vee p_{i-1,j}\vee x_i,\quad
  \neg x_i\vee\neg p_{i-1,j-1}\vee s_{i,j},\quad \neg s_{i,j}\vee p_{i-1,j-1}
  \]
  express $s_{i,j}\leftrightarrow p_{i-1,j}\vee(x_i\wedge p_{i-1,j-1})$. For $j=1$ the last clause is trivially true and omitted. The family is completed by the units $\neg s_{1,j}$ for $2\le j\le e$, the unit $s_{325,e}$, and the clauses $\neg p_{i,e}\vee\neg x_{i+1}$ for $1\le i\le324$.
\end{itemize}
The unit $s_{325,e}$ therefore forces at least $e$ edges. The last clauses forbid an edge once $e$ edges have been counted, so exactly $e$ edges remain. Conversely, every assignment with exactly $e$ true edge variables extends to a model by giving each register its intended value, and this extension is unique on the variables that occur in $F_e$.

\begin{lemma}\label{lem:encoding}
For every graph $G$ on $V$, the assignment of $G$ extends to a model of $F_e$ if and only if
$\alpha(G)\le5$, $\omega(G)\le5$, and $\ee(G)=e$. The extension is unique on the variables in $F_e$.
\end{lemma}

\begin{table}[h]\centering
\caption{Formula sizes for the searched edge counts. $F_e$ has $325\times(e+1)$ variables and
$460{,}460+1297\times e$ clauses. The per-cell columns are maxima over the cells of an edge count, except
the mean number of symmetry clauses. The augmented formula of a cell is $F_e$ with its cube, restricted
local-density and symmetry breaking clauses. $|D|$ is the number of distinct $6$-sets in its restricted local-density clauses, one densest $6$-set per base graph; it is at most the number of base graphs $\mathcal G$ of line~$6$ of Algorithm~\ref{alg:pipeline}, and smaller when several base graphs share a densest $6$-set. The five analytic-band cells $e\in[55,59]$ have between $18{,}200$
and $19{,}500$ variables, and their corresponding formulas $F_e$ have between $531{,}795$ and $536{,}983$ clauses.}
\label{tab:sizes}
\setlength{\tabcolsep}{4pt}
\begin{tabular}{rrrrrrrr}
\toprule
 & & & & & \multicolumn{2}{c}{symmetry clauses} & augmented\\
\cmidrule(lr){6-7}
$e$ & variables & clauses of $F_e$ & cells & $|D|$ & mean & max & clauses\\
\midrule
60 & $19{,}825$ & $538{,}280$ & $1$ & $5$ & $33{,}274$ & $33{,}274$ & $573{,}829$\\
61 & $20{,}150$ & $539{,}577$ & $1$ & $7$ & $43{,}150$ & $43{,}150$ & $585{,}912$\\
62 & $20{,}475$ & $540{,}874$ & $1$ & $25$ & $52{,}122$ & $52{,}122$ & $604{,}371$\\
63 & $20{,}800$ & $542{,}171$ & $4{,}096$ & $13$ & $17{,}713$ & $44{,}180$ & $586{,}363$\\
64 & $21{,}125$ & $543{,}468$ & $4{,}096$ & $28$ & $21{,}860$ & $54{,}983$ & $598{,}918$\\
65 & $21{,}450$ & $544{,}765$ & $4{,}096$ & $56$ & $27{,}145$ & $70{,}440$ & $617{,}665$\\
\bottomrule
\end{tabular}
\end{table}

The counter contributes $325\times e$ auxiliary variables (including the unused $s_{1,1}$) and $1297\times e$ clauses, so $F_e$ has $325\times(e+1)$
variables and $460{,}460+1297\times e$ clauses (Table~\ref{tab:sizes}). A cell of the search adds up to
three further layers to $F_e$, and the result is the cell's augmented formula refuted by CaDiCaL in the end
(Section~\ref{sec:cert}).
\begin{itemize}
\item[(U)] \emph{Cube.} For $e\ge63$, the $12$ unit clauses fixing the cube variables
  $B$ (chosen in Section~\ref{sec:cubevars}), namely $x_{222}$, $x_{159}$, $x_{175}$, $x_{160}$, $x_{174}$, $x_{208}$, $x_{176}$, $x_{235}$, $x_{189}$, $x_{172}$, $x_{220}$, $x_{255}$, the variables of the pairs $\{11,13\}$, $\{7,12\}$, $\{8,11\}$, $\{7,13\}$, $\{8,10\}$, $\{10,13\}$, $\{8,12\}$, $\{12,13\}$, $\{8,25\}$, $\{7,25\}$, $\{10,25\}$, $\{13,21\}$ in this order, to one of their $2^{12}=4096$ truth assignments. The $4096$ cubes
  are pairwise disjoint and together cover all assignments.
\item[(L$_D$)] \emph{Restricted local density.} For a set $D$ of $6$-sets, the clauses
  $\bigvee_{x\in Y}\neg x$ for every $S\in D$ and every $12$-element $Y\subseteq X(S)$: at most
  $11$ of the $15$ edges inside $S$. We write $L_D$ for this set of clauses. These are $\binom{15}{12}=455$ clauses per $6$-set, each
  implied by property (L), so the layer excludes no graph that satisfies (L). The \emph{base graphs} of a cell are the graphs
  that SMS emits for $F_e$ and the cell's cube (Section~\ref{sec:enum}). For each base graph, $D$
  contains one $6$-set on which that graph has the most edges. No base graph satisfies (L)
  (Table~\ref{tab:counts} and Section~\ref{sec:cert}), so each of these $6$-sets spans at least $12$ edges, and the layer excludes
  every base graph.
\item[($\Sigma$)] \emph{Symmetry.} The symmetry-breaking clauses that SMS emits while solving the
  cell, each validated against its permutation witness (Section~\ref{sec:cert}).
\end{itemize}

\paragraph{The sparse-versus-dense choice.}
Generating $G$ rather than its dense complement $H$ is the key modeling choice. With $\alpha\le5$ and at most $65$ edges, every graph of the \emph{family} that $F_e$ describes is sparse. In our preliminary runs, the per-graph minimality check was efficient, producing $0.3$ to $0.5$ graphs per second. But the dense complement has about $260$ edges, and the minimality check of SMS on dense graphs with $26$ vertices is so costly that a direct enumeration of $H$ produced no graphs within $7$ minutes.

\subsection{Enumeration and Cube-and-Conquer}\label{sec:enum}

For each $e$, we run SMS and test property (L) on each emitted graph. The test computes over all $\binom{26}{6}$ $6$-sets the maximum number of spanned $G$-edges and reports any graph whose maximum is at most $11$. The counts $e\in\{60,61,62\}$ are enumerated monolithically. For
$e\in\{63,64,65\}$ the search is partitioned by cube-and-conquer
\citep{HeuleKullmannWieringaBiere}: we fix $12$ edge variables in all $2^{12}=4096$ ways, splitting the formula into disjoint
subproblems whose union is the full search space. Each cube runs as a separate single-core task. A \emph{cell} is one unit of the search: the monolithic formula of a count $e\le62$, or one cube of a partitioned count. The selection of the cube variables is presented in Section~\ref{sec:cubevars}. The graph enumeration is complete, and the guarantee attaches to the union of the cubes. In the SMS propagator, the minimality check compares a graph with all its relabelings, including those that violate the cube. Therefore, a single cube need not emit every graph that satisfies it, as only the lexicographically minimal representatives of the class are emitted. Since the $4096$
cubes are pairwise disjoint, every isomorphism class of the family has its canonical graph in exactly one cube. The union of the cubes has at least one representative of every class (line~$6$ of Algorithm~\ref{alg:pipeline}). Afterwards, property (L) is tested on the emitted graphs. In Table~\ref{tab:counts}, duplicate emissions only inflate the counts without affecting the final number of emitted graphs that satisfy property (L), which is $0$.

\begin{table}[h]\centering
\caption{Results for the searched edge counts $e\in[60,65]$. Cells: $1$ for a monolithic count, $4096$
cubes for a partitioned one. Emitted: the base graphs, i.e., the graphs with $\alpha\le5$, $\omega\le5$
and $e$ edges that SMS emits, repetitions included; sat.~(L): those satisfying property (L).
Times are in core-hours (core-h), hours of single-threaded tasks. The certification columns cover the
full pipeline of Section~\ref{sec:cert} per cell, from the enumeration to the \texttt{drat-trim}
check: the wall-clock of the slowest cell, the total time, and the total size of the raw DRAT
refutations. All times come from the certification run.}\label{tab:counts}
\setlength{\tabcolsep}{4pt}
\begin{tabular}{rrrrrrrr}
\toprule
 & & \multicolumn{3}{c}{enumeration} & \multicolumn{3}{c}{certification}\\
\cmidrule(lr){3-5}\cmidrule(lr){6-8}
$e$ & cells & emitted & sat.\ (L) & core-h & slowest (h) & core-h & DRAT (GB)\\
\midrule
60 & $1$ & $124$ & $0$ & $0.3$ & $0.9$ & $0.9$ & $1.0$\\
61 & $1$ & $503$ & $0$ & $0.5$ & $1.2$ & $1.2$ & $0.9$\\
62 & $1$ & $2{,}455$ & $0$ & $0.8$ & $1.6$ & $1.6$ & $1.3$\\
63 & $4{,}096$ & $26{,}521$ & $0$ & $548$ & $5.5$ & $1{,}959$ & $1{,}296$\\
64 & $4{,}096$ & $238{,}420$ & $0$ & $772$ & $6.7$ & $2{,}799$ & $1{,}760$\\
65 & $4{,}096$ & $2{,}605{,}003$ & $0$ & $1{,}219$ & $9.8$ & $4{,}143$ & $2{,}484$\\
\midrule
total & $12{,}291$ & $2{,}873{,}026$ & $0$ & $2{,}540$ & & $8{,}905$ & $5{,}543$\\
\bottomrule
\end{tabular}
\end{table}

Table~\ref{tab:counts} reports, for each searched edge count, the number of graphs emitted and the number satisfying property (L). The latter number is always $0$. Therefore, no graph on $26$ vertices with $e\in[60,65]$ edges satisfies both $\alpha\le5$ and property (L). Together with Section~\ref{sec:analytic}, this proves the Main Lemma. The remaining columns report the cost: the enumeration takes $2{,}540$ core-hours in total, and the certification of Section~\ref{sec:cert}, whose first stage repeats this enumeration, takes $8{,}905$.

\subsection{Choosing the Cube Variables}\label{sec:cubevars}

Any set $B$ of $k$ edge variables partitions the search space into $2^k$ cubes, but the workload is not guaranteed to spread evenly. Because SMS keeps mostly lexicographically minimal graphs, most edge variables are nearly constant over the family. In the preliminary experiments, a run for $e=62$ that was stopped early had saved $2{,}444$ of the $2{,}455$ graphs. On these graphs, we found $236$ of the $325$ edge variables occur in $<5\%$ of the graphs, and $47$ variables occur in $>95\%$. Surprisingly, only one variable lies in $[35\%,65\%]$. Fixing a nearly constant variable contributes almost nothing to the split. Therefore, we take $B$ to be the $k$ edge variables whose presence frequency is closest to $\tfrac12$ on the family for $e=62$. Later, we use the same $B$ for $e\in\{63,64,65\}$.

To verify the effectiveness of the mentioned variable selection method, we compare it with other alternatives on the same family. Table~\ref{tab:cubevars} shows the number and the share of the emitted graphs in the heaviest cube. This ratio of graphs is a proxy for load balance that is available before any cube has run (in the final search $e=62$ is a single cell; see Table~\ref{tab:counts}).
Running the cubes of every candidate would cost one cube-and-conquer run per candidate, so we compute these numbers from the saved graphs. A saved graph sets all $325$ edge variables, so its edges show which cube it satisfies. In every cube set of Table~\ref{tab:cubevars}, each graph satisfies exactly one cube, so counting the graphs per cube needs no solver run on the cubes. The heaviest cube matters because it shows the imbalance of the load among cubes. From the share in Table~\ref{tab:cubevars}, the default SMS cubing puts almost all emitted graphs into a single cube, and
the matchings do little better, leaving three quarters of the family in one cube, and the first twelve edges at one vertex do not split it at all. However, the frequency-based selections bring the heaviest share down to $15.1\%$, $8.3\%$ and $7.0\%$ for $k=10$, $12$ and $14$. We took $k=12$ for later experiments: its $4096$ cubes give about $14$ cubes for each of the $288$ cores that ran at once, while $k=14$ would quadruple the number of cubes for a gain of $1.3$ points. We also tried the look-ahead cube generator \texttt{march\_cu}~\cite{HeuleKullmannWieringaBiere}, which was found to be the most suitable cuber for SMS~\citep{KirchwegerXiaPeitlSzeider}. It did not return cubes within $300$ seconds on the full formula for $e=65$. An SMS pre-run removed only $13\%$ of its clauses, since the $6$-set clauses do not simplify away.

\begin{table}[h]\centering
\caption{Cube-variable choices compared on one family: $2{,}444$ graphs with $\alpha\le5$, $\omega\le5$ and $62$ edges. Heaviest cube: the number and share of these graphs in the most populated cube. Nonempty: the number of cubes that contain at least one of these graphs. A matching is a set
of pairwise disjoint edges. SMS's default cubing (SMS option \texttt{-{}-simple-assignment-cutoff~30}) emits the
current partial assignment as a cube once at least $30$ edge variables are assigned. The $10{,}246$ cubes fix $30$ to $57$ edge variables each. The SMS cubes come from our first cubing run on the same family.}
\label{tab:cubevars}
\setlength{\tabcolsep}{4pt}
\begin{tabular}{lrrrr}
\toprule
cube variables & $k$ & cubes & nonempty & heaviest cube\\
\midrule
SMS default cubing & -- & $10{,}246$ & $2$ & $2{,}443$ ($99.96\%$)\\
edges of an $8$-edge matching & $8$ & $256$ & $5$ & $1{,}899$ ($77.7\%$)\\
edges of a $12$-edge matching & $12$ & $4{,}096$ & $8$ & $1{,}812$ ($74.1\%$)\\
first $12$ edges at vertex $0$ & $12$ & $4{,}096$ & $1$ & $2{,}444$ ($100\%$)\\
frequency nearest $\tfrac12$ & $10$ & $1{,}024$ & $122$ & $368$ ($15.1\%$)\\
frequency nearest $\tfrac12$ (chosen) & $12$ & $4{,}096$ & $206$ & $203$ ($8.3\%$)\\
frequency nearest $\tfrac12$ & $14$ & $16{,}384$ & $353$ & $172$ ($7.0\%$)\\
\bottomrule
\end{tabular}
\end{table}

The computation of this section and Section~\ref{sec:cert} confirms the choice at the counts. For $e=63$, $64$ and $65$ the heaviest cube held $8.1\%$, $9.8\%$ and $9.7\%$ of the family, and $3{,}574$, $3{,}074$ and $2{,}393$ of the $4096$ cubes emitted no graph. Running time is driven by search rather than by graph count. These empty cubes account for $85\%$, $70\%$ and $50\%$ of the certification time in Table~\ref{tab:counts}, and the slowest cube for $e=63$ is empty. In the first complete enumeration run, a separate earlier run under different cluster load than the certification run of Table~\ref{tab:counts}, with one cube per task, the counts $e=63$, $64$ and $65$ finished in
$3.0$, $3.5$ and $10.3$ hours of wall-clock time with at most $288$ single-threaded tasks at once.

\section{Certification}\label{sec:cert}

For the certification, we first check the encoding as described below, and then we certify the search in three levels of formality. Table~\ref{tab:trust} presents what each level establishes.

\begin{table}[h]\centering
\caption{The trusted base, component by component. Every computational step specific to this proof is
checked by a program separate from the search: Lean~4, \texttt{drat-trim}, the witness checker, or the
encoding checks of this section. The result depends on neither external theorem:
the certified search independently covers the analytic dispatch, and the lower
bound of $55$ edges is proved in Lean. The certificate and encoding checkers are separate from the search programs but
are not themselves formally verified.}
\label{tab:trust}
\setlength{\tabcolsep}{3pt}
\begin{tabular}{ll}
\toprule
proof component & established by\\
\midrule
reduction; $\alpha\le5$, (L) isomorphism-invariant; $\ee(G)\ge55$ & Lean~4\\
base formulas $F_e$, $e\in[55,65]$ & regenerated (Lemma~\ref{lem:encoding})\\
$e\in[55,59]$ ruled out & analytic $+$ \texttt{drat-trim}\\
$e\in[60,65]$ complete & SMS, completeness-checked re-run\\
augmented formula unsatisfiable, each cell & \texttt{drat-trim}\\
symmetry clauses, each cell & witness check\\
Brouwer, Kang--Pikhurko theorems & external\\
\bottomrule
\end{tabular}
\end{table}

\paragraph{Encoding checks.}
In the computation, the base formulas must match the formulas $F_e$ described in Section~\ref{sec:encoding}. To this end, we design an encoder independent from \textsc{Encode} of Algorithm~\ref{alg:pipeline}. The encoder rebuilds the clause groups (I), (K) and (E) from their definitions. For each of the eleven base encoding files, identified by SHA-256 hashes, we checked that the clause multiset and the variable/clause counts agree with the generated formula. Next, we injected three flaws (a dropped clause, a flipped literal and an added clause), and all of them were detected. In addition, the cube file contains each of the $2^{12}$ truth assignments of the $12$ cube variables exactly once. Each clause of a restricted local-density layer is one of the $455$ at-most-$11$ clauses of a $6$-set and is therefore implied by (L). In every cell with $e\in[60,65]$, the recorded counts satisfy $|L_D|=455\,|D|$ and sum to the size of the augmented formula. The only cell whose files were all retained ($e=63$, cube~$3228$ of the cube file) also passes the checks. All $753$ cells that were run more than once produced byte-identical augmented formulas.

\paragraph{A completeness-checked re-run.}
The first level addresses the basic risk that a subproblem was silently truncated and reported complete.
We re-ran the entire enumeration under a stricter runner that requires, for each monolithic cell and
each of the $3\times4096$ cubes, a four-part completeness certificate: (1)~SMS exited with
status $0$, so it was not killed, timed out, or crashed. (2)~It printed its closing timing line, which
it prints only at a normal end. (3)~Its SAT search ended with result code $20$ (unsatisfiable), meaning
the search tree was exhausted rather than cut off. (4)~The number of graphs SMS reported equals the number of graphs captured from its output. A separate verifier confirms that every cube is present exactly once with its certificate. The re-run reproduces Table~\ref{tab:counts} exactly.

\paragraph{Clausal certificates from proof-logging SMS.}
The second level gives each cell a clausal certificate that an independent program checks. A cell
claims that no graph $G$ in the cell satisfies $\alpha(G)\le5$, $\omega(G)\le5$, $\ee(G)=e$ and (L).
The formula of a cell consists of $F_e$, the unit clauses of its cube (for $e\ge63$) and the layer
$L_D$. Its certificate is built in three steps. First, the proof-logging SMS~\citep{KirchwegerManriqueSzeider} runs on the cell and outputs each symmetry clause it learns,
together with a permutation witness. Second, these clauses are added to the formula of the cell, and
CaDiCaL~\citep{BiereCaDiCaL2024} refutes the resulting augmented formula. Third,
\texttt{drat-trim}~\citep{WetzlerHeuleHunt} checks the refutation, which CaDiCaL writes in the standard
DRAT proof format.

The formula of a cell contains (L) only in restricted form. The full (L) encoding with the at-most-$11$ constraints on all $\binom{26}{6}$ $6$-sets would add about $10^8$ clauses per cell. Therefore, we add the constraint only on the $6$-sets of $D$ (Section~\ref{sec:encoding}), one densest $6$-set per base graph. Every clause of this layer $L_D$ is implied by (L), and the layer is then small (Table~\ref{tab:sizes}).

A refutation shows only that the augmented formula is unsatisfiable. It proves the claim of the cell as the added clauses exclude no canonical graph with the target properties ($\alpha\le5$, $\omega\le5$, $e$ edges and (L)). Supposing some graph $G$ has these properties, they are invariant under isomorphism (the Lean development proves this for $\alpha\le5$ and (L)), so the canonical copy of $G$ has them as well. The canonical copy also lies in one cell, and the corresponding assignment extends to a model of $F_e$ (Lemma~\ref{lem:encoding}). It satisfies $L_D$, because every clause of $L_D$ is implied by (L), and it satisfies every symmetry clause as checked below. The augmented formula of that cell would then be satisfiable, which its refutation rules out.

Each symmetry clause comes with a permutation witness $\sigma$. Let $M$ be the $0/1$ vector of the
$325$ edge variables in the order $x_1,\dots,x_{325}$ of Section~\ref{sec:encoding}, and let
$\sigma(M)$ be this vector for the graph relabeled by $\sigma$. Every canonical graph satisfies the
\emph{lex-leader inequality} $M\le_{\mathrm{lex}}\sigma(M)$, where $\le_{\mathrm{lex}}$ is the
lexicographic order of Section~\ref{sec:prelim}. Therefore, a clause excludes no canonical graph if every assignment falsifying it violates this inequality. Our checker tests this condition. A fast test looks for one position where every such assignment makes $M$ and $\sigma(M)$ first differ, with
$1$ in $M$ and $0$ in $\sigma(M)$. If no such position exists, an exact test considers each way to
satisfy the inequality: $M=\sigma(M)$, or a first difference at a position where $M$ has $0$ and
$\sigma(M)$ has $1$. For both cases, union-find on the equalities and fixed values that the case
forces decides whether the case is consistent with falsifying the clause, and the clause is accepted
when no case is. The checker shares no code with the SMS build that emitted the clauses. It accepts all
$273{,}406{,}587$ symmetry clauses of the $12{,}291$ cells.

We ran this pipeline (Algorithm~\ref{alg:pipeline}, lines 8--10) on all $12{,}291$ cells: the three
monolithic cases and the $3\times4096$ cubes. Every refutation passes \texttt{drat-trim} without
failure. The cells exclude $2{,}873{,}026$ base graphs in total, and the per-cell counts reproduce
Table~\ref{tab:counts}. Table~\ref{tab:counts} also gives the time and the proof size per edge count. In detail, the time cost splits nearly evenly over the four stages of a cell: $29\%$ for enumerating the base graphs, $26\%$ for the
proof-logging SMS run that emits the symmetry clauses, $23\%$ for the CaDiCaL refutations, and $22\%$ for checking them with \texttt{drat-trim}.

The raw DRAT proofs total $5.5$\,TB, or $2.0$\,TB after \texttt{drat-trim} removes the steps unused by the refutation. Therefore, for every cell, we store a complete record of the run instead of the proof itself (times, SHA-256 hashes of all artifacts, and solver statistics). The solvers are deterministic, so any cell's certificate regenerates from the recorded inputs, identified by their hashes, and re-checks against its recorded hash.

\paragraph{The analytic edge counts, certified.}
The five near-extremal counts, $e\in[55,59]$, are settled analytically in Section~\ref{sec:analytic}, and they were also
put through the certificate pipeline. For $e\in[55,58]$ the family of graphs with $\alpha\le5$, $\omega\le5$ and $e$ edges is empty, so the cell has no base graphs. When $D$ is empty, its formula is directly refuted. For $e=59$, SMS emits seven graphs, but none of them satisfies (L). These seven graphs are copies of exactly two isomorphism classes, the complements of the graphs of Theorem~\ref{thm:kp} for the tuple $(5,5,5,5,5)$ with $q=1$ and $q=2$, which is precisely the family that the argument of Section~\ref{sec:analytic} predicts for $\alpha\le5$, $\omega\le5$ and $59$ edges, so the enumeration cross-checks Theorem~\ref{thm:kp}, as the empty families for $e\in[55,58]$ cross-check Theorem~\ref{thm:brouwer}. All five augmented refutations pass \texttt{drat-trim}, and all $97{,}351$ of their symmetry clauses are validated against their permutation witnesses. With them the certification comprises $12{,}296$ cells and $273{,}503{,}938$ witness-checked clauses in total. In the end, every edge count in $[55,65]$ carries a \texttt{drat-trim}-checked refutation whose symmetry clauses are witness-checked.

\paragraph{A Lean 4 formalization of the reduction, and its contract.}

The third level formalizes the reduction of Section~\ref{sec:reduction} in Lean~4 with Mathlib
\citep{Mathlib}. A formal theorem is only as good as the fit between its statement and the claim, so
we checked its \emph{contract}: the definitions the statements use and the statements themselves are the same as they appear in the development. (See Listing~\ref{lst:contract} in Appendix~\ref{app:lean}). A $5$-coloring of $K_{26}$ is a map \lstinline{c : Sym2 V → Fin 5} on a type \lstinline{V} with $26$ vertices.
The values on the diagonal pairs \lstinline{s(v, v)} are never consulted. Here \lstinline{Sym2 V} is Mathlib's type of unordered pairs \lstinline{s(v, w)} of vertices of \lstinline{V}, diagonal pairs
included, and \lstinline{Fin 5} is the type $\{0,\dots,4\}$, so Lean numbers the colors $0,\dots,4$ rather than $1,\dots,5$. The color class \lstinline{k} is
the simple graph \lstinline{classGraph c k} of the \lstinline{k}-colored pairs. \lstinline{Balanced},
\lstinline{NoIndep6} and \lstinline{PropertyL} literally state the definitions of Sections~\ref{sec:intro}
and~\ref{sec:reduction}, except that \lstinline{PropertyL} counts ordered pairs and is therefore bounded by $22=2\cdot11$.

The contract reads against the paper as follows.
\begin{itemize}
\item \emph{Notation.} \lstinline{G.edgeFinset.card} is the number of edges of \lstinline{G},
  \lstinline{S.offDiag} is the set of ordered pairs of distinct elements of \lstinline{S}, and
  \lstinline{f : G ≃g G'} is a graph isomorphism. The argument \lstinline{[DecidableRel G.Adj]} only
  makes adjacency decidable, so that the edges can be counted.
\item \lstinline{noIndep6_classGraph} and \lstinline{propertyL_classGraph} are properties (A) and (L)
  of Section~\ref{sec:reduction}, and \lstinline{exists_class_le_65} is the averaging bound of
  Lemma~\ref{lem:band}.
\item \lstinline{card_edgeFinset_ge_55} is the lower bound of Lemma~\ref{lem:band}. Lean proves it
  from the formalization of Tur\'an's theorem in Mathlib applied to the complement, which gives the
  same bound as the Caro--Wei inequality.
\item \lstinline{no_balanced_5_colouring_of_mainLemma} concludes \lstinline{¬ Balanced c} for every
  coloring of a $26$-element type, which is Theorem~\ref{thm:main}. Its only hypothesis,
  \lstinline{step_sms_full}, is a statement about graphs alone. It is the Main Lemma, which the certified search establishes. The search also imposes $\omega\le5$, but (L) implies it because a $K_6$ has $15$ edges.
\item The variant \lstinline{no_balanced_5_colouring} of \texttt{Basic.lean}, not shown in
  Listing~\ref{lst:contract}, predates \lstinline{card_edgeFinset_ge_55} and has four hypotheses in place of the Main Lemma: the lower bound, which is now proved, and one per case of Fig.~\ref{fig:roadmap}, namely Brouwer,
  Kang--Pikhurko, and the search for $e\in[60,65]$. The lemmas \lstinline{noIndep6_iso} and \lstinline{propertyL_iso} give the
  isomorphism invariance that the certification argument above relies on.
\item \emph{Checks.} With Lean~4.28.0 and Mathlib, the development builds without errors, and \lstinline{#print axioms} reports only \lstinline{propext}, \lstinline{Classical.choice} and \lstinline{Quot.sound}, the standard axioms of Lean's logic, for every theorem in the listing. Therefore, the development adds no axioms, and its one remaining hypothesis is the Main Lemma, which the certified search establishes outside Lean. Lean prints the conclusion of the main theorem as \lstinline{¬ErdosGyarfasR5.Balanced c}, and it can confirm that it refers to the definition above and not to Mathlib's unrelated \lstinline{Balanced}, a property of sets in normed spaces.
\end{itemize}

\paragraph{Artifact and reproducibility.}
The anonymous artifact is available~\citep{Artifact}. (Note that the link may be broken when copied with multiple lines, so please copy the links if clicks fail.)
The encodings, the cube partition, the property (L) test, the completeness-checked runner and
verifier, the certificate pipeline, the per-cell run records with the SHA-256 hashes of every
artifact, complete certificate chains for representative cells,
the regeneration and re-verification tooling, and the Lean~4 development accompany the paper as an anonymized code-and-data archive in the supplementary material. The full
witness and certificate corpus exceeds the supplementary size limit, so its archives are identified by
SHA-256 hashes in the supplement and regenerate deterministically from the recorded inputs. The archive
also records the checksums of the binaries and the hashes of every input formula and cube file, and the
Lean~4 development builds against a fixed Mathlib version. All runs used cluster nodes with two $24$-core AMD EPYC~7402 CPUs ($96$ hardware threads) and
$1$\,TB of RAM. The search and its certification ran under Ubuntu~18.04 (Linux~$4.15$), up to $288$ single-threaded
tasks concurrently, and the direct-solving tests of Section~\ref{sec:intro} under Debian~13
(Linux~$6.12$), kissat on one core and CP-SAT with eight threads. During development two settings were explored, the cube variables and their number $k$ (Section~\ref{sec:cubevars} reports all candidates and the selection criterion). The remaining settings (the per-tool time limits, the solver versions, and the minimality-check cutoff of SMS, left at its default of $200{,}000$ steps in the enumeration and switched off (\texttt{-{}-cutoff~0}) in the proof-logging runs that emit the symmetry clauses) were fixed at the beginning. All of these affect only the cost of the search, because the validity of the result rests on the certificates as we discussed, which are checked independently of how the search was configured.

\section{Discussion}\label{sec:discussion}

\paragraph{Validation.}
Three checks support the computation. First, the same reduction and enumeration, without the
certificate stages, reproduce the \emph{known} case $r=4$: over the band $\ee(G_0)\in[28,34]$, the search emits
every isomorphism class on $17$ vertices with $\alpha\le4$ and $\omega\le4$, and
none satisfies the $r=4$ local-density bound that every $5$-set spans at most $7$ edges, reproducing
the nonexistence of a balanced $4$-coloring of $K_{17}$. Second, the direct encoding of a balanced coloring as one propositional formula
(Section~\ref{sec:intro}) passes positive and
negative controls: it finds balanced colorings of $K_9$ and $K_{16}$ and reports
nonexistence for $K_{10}$, and for $K_{17}$ through the reduction used in the first check. Third, an independent simulated-annealing search (a randomized local search) found a
genuinely balanced $5$-coloring of $K_{25}$ that is \emph{not} affine: its color classes have
$56$, $59$, $60$, $61$ and $64$ edges, not the affine $50$, $50$, $50$, $50$ and $100$. Consistent with
Theorem~\ref{thm:main}, it does not extend to $K_{26}$.

\paragraph{The method.}
The proof instantiates a reusable template for extremal existence questions: \emph{reduce} the
existence question to a property of one sparse substructure with parameters in a narrow
range; \emph{dispatch} the extreme parameter values
analytically with stability theorems, which apply precisely at near-extremal values; and
\emph{enumerate} the remaining mid-range values exhaustively with SAT modulo symmetries. The
reduction grew from one observation: once one color class is fixed, deciding whether the remaining
edges can be colored to complete a balanced coloring is essentially instant already at $r=4$, while the single balanced-coloring formula is out of reach. The hard part
is ranging over all candidate classes, an isomorph-free generation problem.

\paragraph{Related work.}
The Boolean Pythagorean triples problem \citep{HeuleKullmannMarek} and Schur number five
\citep{Heule2018} were settled by massive cube-and-conquer refutations with DRAT certificates, and
in both the solver carries essentially the entire case distinction, with no mathematical reduction
in front of it. Other major
certificate-backed SAT computations include the resolution of Keller's conjecture
\citep{BrakensiekHeuleMackeyNarvaez}. More recently, the empty-hexagon number
\citep{SubercaseauxEtAl2024} combines SAT certification with a reduction formalized in an interactive theorem prover. SAT modulo
symmetries \citep{KirchwegerSzeider} brought isomorph-free generation into the solver, with
cube-and-conquer strategies tailored to it \citep{KirchwegerXiaPeitlSzeider} and a formally verified
extension that checks the symmetry reasoning in Lean \citep{KirchwegerManriqueSzeider}. A side condition like property (L), whose violation is witnessed by a single $6$-set spanning at least $12$ edges, can instead be
enforced inside the solver by co-certificate learning \citep{KirchwegerPeitlSzeider2023}, which turns
each such witness into a learned clause, but
generating the base family and imposing a sound restricted (L) is simpler at our scale. Pairing an analytic reduction
with a certified search is an established pattern: the nonexistence of a projective
plane of order $10$ was settled this way by \citet{LamThielSwiercz1989} and re-established with a
SAT certificate \citep{BrightEtAl2021}, and isomorph-free generation of the sparse side is standard
since \citet{McKay1998}. What is specific here is that both halves are checked: the reduction in Lean,
and every cell by an independent refutation.

\section{Conclusion}\label{sec:conclusion}

We settled the Erd\H{o}s--Gy\'arf\'as balanced-coloring conjecture for $r=5$, the first new case since 1999: $K_{25}$ is the largest complete graph with a balanced $5$-coloring. The proof reduces the conjecture to an extremal lemma about one color class, dispatches the near-extremal counts analytically, and settles the rest by a certified enumeration, with every computational step checked by a program separate from the search. The next case, $r=6$, is of independent interest: $6$ is not a prime power, so no affine plane of order $6$ exists, and it is open whether even $K_{36}$ has a balanced $6$-coloring.

\begin{credits}
\subsubsection{Acknowledgments}
The authors acknowledge the support by the European Union's Horizon 2020 research and innovation programme
under the Maria Skłodowska-Curie grant agreement
No.~101034440\marginpar{\worldflag[width=18pt]{EU}}, and the support from the Austrian Science Fund
  (FWF), project 10.55776/P36688, 10.55776/P36420 and 10.55776/COE12.

\subsubsection{Use of AI assistants}
We used generative AI tools to implement the Lean formalizations. We also used AI tools to draft and edit text and figures. We reviewed all AI-assisted work. Code and proofs are checked by tests and the Lean kernel. We take responsibility for the final content of this paper, including text, claims and artifacts produced with the aid of generative AI.
\end{credits}

\bibliographystyle{plainurl}
\bibliography{egr5}

\begin{thebibliography}{10}

\bibitem{Artifact}
Anonymous.
\newblock Anonymized research artifact (code, proofs and data).
\newblock Zenodo, 2026.
\newblock
  \href{https://zenodo.org/records/21246738?preview=1&token=eyJhbGciOiJIUzUxMiJ9.eyJpZCI6IjVlMTBlOWNhLWVjYzctNDAyNS05MTU0LTlmYjIzODZlZDE5OCIsImRhdGEiOnt9LCJyYW5kb20iOiIzMzg2Njg2YmI5NGVlNDM4NzEzODliNGMyZDUxYTBhZiJ9.R1z8-3vEo88uwNzBH0D32tfbd0fnsHlE9HW7e8L1Tr-2qJG2Z86mZs9oxg2NiNl-7CNijSoXkPhanKMrdwJ8eQ}{Access
  link for reviewers}.

\bibitem{BiereCaDiCaL2024}
Armin Biere, Tobias Faller, Katalin Fazekas, Mathias Fleury, Nils Froleyks, and
  Florian Pollitt.
\newblock {CaDiCaL} 2.0.
\newblock In {\em Computer Aided Verification (CAV)}, volume 14681 of {\em
  Lecture Notes in Computer Science}, pages 133--152. Springer, 2024.

\bibitem{BiereKissat2020}
Armin Biere, Katalin Fazekas, Mathias Fleury, and Maximilian Heisinger.
\newblock {CaDiCaL}, {K}issat, {P}aracooba, {P}lingeling and {T}reengeling
  entering the {SAT} {C}ompetition 2020.
\newblock In Tom{\'a}{\v s} Balyo, Nils Froleyks, Marijn Heule, Markus Iser,
  Matti J{\"a}rvisalo, and Martin Suda, editors, {\em Proceedings of SAT
  Competition 2020: Solver and Benchmark Descriptions}, volume B-2020-1 of {\em
  Department of Computer Science Report Series B}, pages 50--53. University of
  Helsinki, 2020.

\bibitem{BrakensiekHeuleMackeyNarvaez}
Joshua Brakensiek, Marijn Heule, John Mackey, and David~E. Narv{\'a}ez.
\newblock The resolution of {K}eller's conjecture.
\newblock {\em Journal of Automated Reasoning}, 66(3):277--300, 2022.

\bibitem{BrightEtAl2021}
Curtis Bright, Kevin K.~H. Cheung, Brett Stevens, Ilias~S. Kotsireas, and Vijay
  Ganesh.
\newblock A {SAT}-based resolution of {L}am's problem.
\newblock In {\em Proceedings of the 35th AAAI Conference on Artificial
  Intelligence (AAAI)}, pages 3669--3676. AAAI Press, 2021.

\bibitem{Brooks1941}
R.~L. Brooks.
\newblock On colouring the nodes of a network.
\newblock {\em Mathematical Proceedings of the Cambridge Philosophical
  Society}, 37(2):194--197, 1941.

\bibitem{Brouwer1981}
Andries~E. Brouwer.
\newblock Some lotto numbers from an extension of {T}ur{\'a}n's theorem.
\newblock Technical Report ZW 152/81, Mathematisch Centrum, Amsterdam, 1981.
\newblock Zbl 0443.05032.

\bibitem{Caro1979}
Yair Caro.
\newblock New results on the independence number.
\newblock Technical report, Tel-Aviv University, 1979.

\bibitem{ChungLiu1978}
K.~M. Chung and C.~L. Liu.
\newblock A generalization of {R}amsey theory for graphs.
\newblock {\em Discrete Mathematics}, 21(2):117--127, 1978.
\newblock \href {https://doi.org/10.1016/0012-365X(78)90084-5}
  {\path{doi:10.1016/0012-365X(78)90084-5}}.

\bibitem{Crawford1996}
James~M. Crawford, Matthew~L. Ginsberg, Eugene~M. Luks, and Amitabha Roy.
\newblock Symmetry-breaking predicates for search problems.
\newblock In {\em Proceedings of the 5th International Conference on Principles
  of Knowledge Representation and Reasoning (KR)}, pages 148--159. Morgan
  Kaufmann, 1996.

\bibitem{EG1999}
Paul Erd{\H{o}}s and Andr{\'a}s Gy{\'a}rf{\'a}s.
\newblock Split and balanced colorings of complete graphs.
\newblock {\em Discrete Mathematics}, 200(1--3):79--86, 1999.

\bibitem{GreenwoodGleason1955}
R.~E. Greenwood and A.~M. Gleason.
\newblock Combinatorial relations and chromatic graphs.
\newblock {\em Canadian Journal of Mathematics}, 7:1--7, 1955.

\bibitem{HeuleKullmannWieringaBiere}
Marijn Heule, Oliver Kullmann, Siert Wieringa, and Armin Biere.
\newblock Cube and conquer: Guiding {CDCL} {SAT} solvers by lookaheads.
\newblock In {\em Hardware and Software: Verification and Testing (HVC)},
  volume 7261 of {\em Lecture Notes in Computer Science}, pages 50--65.
  Springer, 2011.

\bibitem{Heule2018}
Marijn J.~H. Heule.
\newblock Schur number five.
\newblock In {\em Proceedings of the 32nd AAAI Conference on Artificial
  Intelligence (AAAI)}, pages 6598--6606. AAAI Press, 2018.

\bibitem{HeuleKullmannMarek}
Marijn J.~H. Heule, Oliver Kullmann, and Victor~W. Marek.
\newblock Solving and verifying the {B}oolean {P}ythagorean triples problem via
  cube-and-conquer.
\newblock In {\em Theory and Applications of Satisfiability Testing (SAT)},
  volume 9710 of {\em Lecture Notes in Computer Science}, pages 228--245.
  Springer, 2016.

\bibitem{KangPikhurko2005}
Mihyun Kang and Oleg Pikhurko.
\newblock Maximum {$K_{r+1}$}-free graphs which are not {$r$}-partite.
\newblock {\em Matematychni Studii}, 24(1):12--20, 2005.

\bibitem{KirchwegerManriqueSzeider}
Markus Kirchweger, Pablo Manrique, and Stefan Szeider.
\newblock Formally verified graph generation with {SAT} modulo symmetries and
  {L}ean.
\newblock In Armin Biere, Carsten Lutz, and Sara Negri, editors, {\em Automated
  Reasoning (IJCAR)}, volume 16688 of {\em Lecture Notes in Computer Science},
  pages 117--135. Springer, 2026.
\newblock \href {https://doi.org/10.1007/978-3-032-32589-1_8}
  {\path{doi:10.1007/978-3-032-32589-1_8}}.

\bibitem{KirchwegerPeitlSzeider2023}
Markus Kirchweger, Tom{\'a}{\v{s}} Peitl, and Stefan Szeider.
\newblock Co-certificate learning with {SAT} modulo symmetries.
\newblock In {\em Proceedings of the 32nd International Joint Conference on
  Artificial Intelligence (IJCAI)}, pages 1944--1953, 2023.

\bibitem{KirchwegerXiaPeitlSzeider}
Markus Kirchweger, Tom{\'a}{\v{s}} Peitl, Stefan Szeider, and Hai Xia.
\newblock Smart cubing for graph search: {A} comparative study.
\newblock In {\em Principles and Practice of Constraint Programming (CP)},
  volume 379 of {\em LIPIcs}, pages 33:1--33:19. Schloss Dagstuhl --
  Leibniz-Zentrum f{\"u}r Informatik, 2026.

\bibitem{KirchwegerSzeider}
Markus Kirchweger and Stefan Szeider.
\newblock {SAT} modulo symmetries for graph generation and enumeration.
\newblock {\em ACM Transactions on Computational Logic}, 25(3):1--30, 2024.

\bibitem{LamThielSwiercz1989}
C.~W.~H. Lam, L.~Thiel, and S.~Swiercz.
\newblock The non-existence of finite projective planes of order 10.
\newblock {\em Canadian Journal of Mathematics}, 41(6):1117--1123, 1989.
\newblock \href {https://doi.org/10.4153/CJM-1989-049-4}
  {\path{doi:10.4153/CJM-1989-049-4}}.

\bibitem{McKay1998}
Brendan~D. McKay.
\newblock Isomorph-free exhaustive generation.
\newblock {\em Journal of Algorithms}, 26(2):306--324, 1998.

\bibitem{PerronDidierGay2023}
Laurent Perron, Fr{\'e}d{\'e}ric Didier, and Steven Gay.
\newblock The {CP-SAT-LP} solver.
\newblock In {\em Proceedings of the 29th International Conference on
  Principles and Practice of Constraint Programming (CP)}, volume 280 of {\em
  LIPIcs}, pages 3:1--3:2. Schloss Dagstuhl -- Leibniz-Zentrum f{\"u}r
  Informatik, 2023.

\bibitem{Sinz2005}
Carsten Sinz.
\newblock Towards an optimal {CNF} encoding of boolean cardinality constraints.
\newblock In {\em Principles and Practice of Constraint Programming (CP)},
  volume 3709 of {\em Lecture Notes in Computer Science}, pages 827--831.
  Springer, 2005.

\bibitem{SubercaseauxEtAl2024}
Bernardo Subercaseaux, Wojciech Nawrocki, James Gallicchio, Cayden~R. Codel,
  Mario Carneiro, and Marijn J.~H. Heule.
\newblock Formal verification of the empty hexagon number.
\newblock In {\em Interactive Theorem Proving (ITP)}, volume 309 of {\em
  LIPIcs}, pages 35:1--35:19. Schloss Dagstuhl -- Leibniz-Zentrum f{\"u}r
  Informatik, 2024.

\bibitem{Mathlib}
{The mathlib Community}.
\newblock The {L}ean mathematical library.
\newblock In {\em Proceedings of the 9th ACM SIGPLAN International Conference
  on Certified Programs and Proofs (CPP)}, pages 367--381. ACM, 2020.

\bibitem{Turan1941}
P{\'a}l Tur{\'a}n.
\newblock Eine {Extremalaufgabe} aus der {Graphentheorie}.
\newblock {\em Matematikai {\'e}s Fizikai Lapok}, 48:436--452, 1941.
\newblock In Hungarian, with a German summary. Zbl 0026.26903.

\bibitem{Wei1981}
Victor~K. Wei.
\newblock A lower bound on the stability number of a simple graph.
\newblock Technical Report TM 81-11217-9, Bell Laboratories, 1981.

\bibitem{WetzlerHeuleHunt}
Nathan Wetzler, Marijn J.~H. Heule, and Warren~A. Hunt~Jr.
\newblock {DRAT-trim}: Efficient checking and trimming using expressive clausal
  proofs.
\newblock In {\em Theory and Applications of Satisfiability Testing (SAT)},
  volume 8561 of {\em Lecture Notes in Computer Science}, pages 422--429.
  Springer, 2014.

\end{thebibliography}
\newpage

\appendix
\renewcommand{\thelstlisting}{\arabic{lstlisting}}
\section{The Lean Contract}\label{app:lean}

Listing~\ref{lst:contract} shows the definitions and theorem statements of the Lean development
discussed in Section~\ref{sec:cert}. Lean notation used in it: \lstinline{SimpleGraph V} is a simple
graph on the vertex type \lstinline{V}, with adjacency relation \lstinline{G.Adj}; \lstinline{Finset V}
is a finite set of vertices, \lstinline{S.card} its size, and \lstinline{S.filter p} the elements of
\lstinline{S} satisfying \lstinline{p}; \lstinline{p.1} and \lstinline{p.2} are the components of a pair;
\lstinline{Fintype.card V} is the number of elements of \lstinline{V}; \lstinline{Prop} is the type of
propositions; \lstinline{noncomputable} marks a definition that Lean does not compile to executable code;
and the omitted fields \lstinline{symm} and \lstinline{loopless} of \lstinline{classGraph} prove that
adjacency is symmetric and irreflexive.

\begin{lstlisting}[basicstyle=\ttfamily\scriptsize, caption={The contract of the Lean development: definitions and theorem statements as they appear in the source (\texttt{Basic.lean}, \texttt{EdgeBound.lean}, \texttt{IsoInvariance.lean}), proofs omitted. The vertex types \lstinline{V} and \lstinline{W}, the graphs \lstinline{G} and \lstinline{G'} of the last two lemmas and, where needed, decidable equality and finiteness are implicit section variables.}, label={lst:contract}]
def classGraph (c : Sym2 V → Fin 5) (k : Fin 5) : SimpleGraph V where
  Adj v w := v ≠ w ∧ c s(v, w) = k     -- symm, loopless: proofs omitted
def Balanced (c : Sym2 V → Fin 5) : Prop :=
  ∀ S : Finset V, S.card = 6 → ∀ k : Fin 5, ∃ v ∈ S, ∃ w ∈ S, v ≠ w ∧ c s(v, w) = k
def NoIndep6 (G : SimpleGraph V) : Prop :=
  ∀ S : Finset V, S.card = 6 → ∃ v ∈ S, ∃ w ∈ S, G.Adj v w
noncomputable def edgesOrdWithin (G : SimpleGraph V) (S : Finset V) : ℕ :=
  (S.offDiag.filter (fun p => G.Adj p.1 p.2)).card
def PropertyL (G : SimpleGraph V) : Prop :=
  ∀ S : Finset V, S.card = 6 → edgesOrdWithin G S ≤ 22

theorem noIndep6_classGraph (c : Sym2 V → Fin 5) (hc : Balanced c) (k : Fin 5) :
    NoIndep6 (classGraph c k)
theorem propertyL_classGraph (c : Sym2 V → Fin 5) (hc : Balanced c) (k : Fin 5) :
    PropertyL (classGraph c k)
theorem exists_class_le_65 (c : Sym2 V → Fin 5) (hcard : Fintype.card V = 26) :
    ∃ k : Fin 5, (classGraph c k).edgeFinset.card ≤ 65
theorem card_edgeFinset_ge_55 (hcard : Fintype.card V = 26) (G : SimpleGraph V)
    [DecidableRel G.Adj] (hA : NoIndep6 G) : 55 ≤ G.edgeFinset.card
theorem no_balanced_5_colouring_of_mainLemma
    (step_sms_full : ∀ (G : SimpleGraph V) [DecidableRel G.Adj],
      NoIndep6 G → PropertyL G → ¬ (55 ≤ G.edgeFinset.card ∧ G.edgeFinset.card ≤ 65))
    (hcard : Fintype.card V = 26) (c : Sym2 V → Fin 5) :
    ¬ Balanced c
theorem noIndep6_iso (f : G ≃g G') : NoIndep6 G ↔ NoIndep6 G'
theorem propertyL_iso (f : G ≃g G') : PropertyL G ↔ PropertyL G'
\end{lstlisting}
\end{document}